\documentclass{amsart}

\usepackage{amssymb}
\usepackage{verbatim}
\usepackage{float}
\usepackage{tikz}
\usepackage{tikz-cd}
\usetikzlibrary{arrows,chains,matrix,positioning,scopes}

\usepackage{kbordermatrix}

\input xy
\xyoption{all}

\def\mf{\mathfrak}
\def\ft{\mathfrak{t}}
\def\al{\alpha}
\newcommand{\bbS}{{\mathbb{S}}}
\newcommand{\bbC}{{\mathbb{C}}}

\newtheorem{thm}{Theorem}
\newtheorem{prop}[thm]{Proposition}
\newtheorem{exmp}[thm]{Example}
\newtheorem{rmk}[thm]{Remark}

\numberwithin{equation}{section}
\numberwithin{thm}{section}

\title[Polynomial Assignments]{Polynomial Assignments for\\ Bott-Samelson manifolds}
\date{February 29, 2016.} 

\author{Gouri Shankar Seal}
\address{Department of Mathematics, Northeastern University}
\email{seal.g@husky.neu.edu}

\author{Catalin Zara}
\address{Department of Mathematics, UMass Boston}
\email{catalin.zara@umb.edu}

\begin{document}

\begin{abstract}
Polynomial assignments for a torus $T$-action on a smooth manifold $M$ were introduced in \cite{ggk1}; they form a module over $\bbS(\ft^*)$, the algebra of polynomial functions on $\ft$, the Lie algebra of $T$.
In this paper we describe the assignment module $\mathcal{A}_T(M)$ for a natural $T$-action on a Bott-Samelson manifold $M = BS^I$, and present a method for computing generators.
\end{abstract}

\maketitle

\tableofcontents

\section{Introduction}
The study of polynomial assignments was initiated by Ginzburg, Guillemin, and Karshon (\cite{ggk1}) in conjunction with abstract moment maps, as a method of understanding geometric information  from underlying combinatorial data of group action on manifolds.  Recent works on this topic include \cite{gsz2014} and \cite{gm2015} - the latter is an extension to topological group actions. 

In this paper we describe the assignment module $\mathcal{A}_T(M)$ for a natural action of a torus $T$ on a Bott-Samelson manifold $M = BS^I$ and present a method for computing generators for this $\bbS(\ft^*)$-module, where $\bbS(\ft^*)$ is the algebra of polynomial functions on $\ft$, the Lie algebra of $T$. (See Section~\ref{sec:bs-def} for the definition of a Bott-Samelson manifold $BS^I$ and of the $T$-action on it.) The method has been implemented using Maple $15$, Python and the commutative algebra package Singular.

A Bott-Samelson $T$-space $M=BS^I$ is equivariantly formal with finite fixed point set $M^T$; that allows one to identify $\mathcal{A}_T(M)$ with a subring of $\text{Maps}(M^T, \bbS(\ft^*))$. The conditions that a map that assigns a polynomial to each fixed point must satisfy in order to represent an element of $\mathcal{A}_T(M)$ are encoded in the fixed point data - fixed points and weights of isotropy representations ; we describe that data in Section ~\ref{sec:fixpointdata}. We record the fixed point data in an associated digraph, $\Gamma_I$, with arrows labeled by elements of $\ft^*$; the labels are thought of as homogeneous polynomials of degree one in $\bbS(\ft^*)$ and are in fact roots for the Lie algebra defining the Bott-Samelson manifold.  

In Section~\ref{sec:assigngraph} we describe the combinatorial conditions that an element of $\text{Maps}(M^T, \bbS(\ft^*))$ must satisfy in order to represent an assignment in $\mathcal{A}_T(M)$. Equivariant cohomology classes in $H_T(M)$ can also be represented as maps from $M^T$ to $\bbS(\ft^*)$; the subalgebra of maps that represent equivariant cohomology classes is denoted by $\mathcal{H}_T(M)$. Each equivariant cohomology class defines an assignment, hence $\mathcal{H}_T(M)$ is a subalgebra of $\mathcal{A}_T(M)$; the assignments in  $\mathcal{H}_T(M)$ are called \emph{cohomological}. When the Bott-Samelson manifold is of GKM-type (in this case, when the letters of the word $I$ are distinct), all assignments are cohomological, but in all other cases that is no longer true. The \emph{defect} module $\mathcal{A}_T(M)/\mathcal{H}_T(M)$ is a torsion module, and we compute it in a few cases.

To determine generators for the assignment module $\mathcal{A}_T(BS^I)$ we use an inductive method: $BS^I$ is a bundle over $\bbC P^1$ with fiber $BS^{I'}$, a Bott-Samelson space determined by a word $I'$ that consists of all-but-the-last letter of $I$. Using that, we describe how one can compute the assignment module $\mathcal{A}_T(BS^I)$ from the assignment module $\mathcal{A}_T(BS^{I'})$, first for a particular case (Section~\ref{sec:ex_inductive}) and then in the general case (Section~\ref{sec:gen_construction}). These computations have been implemented in \emph{Maple} by the first author. The assignment modules we have computed are \emph{free} $\bbS(\ft^*)$-modules, and we conjecture that that is the case for all Bott-Samelson manifolds. 

In the last section we describe an alternative construction of generators for the assignment module, based on ideas from Morse theory. A generic $\xi\in \ft$ determines an orientation of the edges of the graph $\Gamma_I$, and the oriented graph is acyclic.  Borrowing ideas from the construction of combinatorial Thom classes, we attempt to determine generators as assignments supported on the flow-up of a fixed point and satisfying natural normalization conditions. The situation is more complicated, as we ended up with vertices having multiple such assignments associated with them, and vertices for which no such assignments are needed. These computations have been implemented in \emph{Singular} (\cite{gp2008}) by the first author.

\emph{Acknowledgments:}  The authors thank Victor Guillemin, Sue Tolman, and  Jonathan Weitsman for many helpful discussions.

\section{Bott-Samelson Spaces}
Bott-Samelson spaces were introduced in \cite{bs1958} as spaces with compact group actions. 
In this section we present an equivalent construction  in the context of complex groups; for details of a related construction see, for example, \cite{g-k1994} or \cite{w2004}.

\subsection{Definition}\label{sec:bs-def}

Let  $G$ be a connected, complex semisimple Lie group with Lie algebra $\mf{g}$, $H$  a Cartan subgroup and $B$ a Borel subgroup of $G$ with $H\subset B \subset G$. 
Let $\Delta \subset \mf{h}^*$ be the set of roots and
\begin{equation}
\mf{g} = \mf{h} +\sum_{\al\in \Delta} \mf{g}_{\al }
\end{equation}
the corresponding Cartan decomposition. Choose a set of positive roots $\Delta^{+}$ and let 
\begin{equation}
\Sigma  = \{ \al_1, \al_2, \ldots, \al_r\} \subseteq \Delta^{+}
\end{equation}
be the simple roots. For an index $i \in [r]:= \{1, 2, \ldots, r\}$, let $P_i :=P_{\alpha_i}$ be the minimal parabolic subgroup associated with $\alpha_i$, and for a word
\begin{equation}
I = [i_1, i_{1}, \ldots, i_d] \in [r]^d, 
\end{equation}
of length $|I| = d$, let
\begin{equation*}
P_I = P_{i_d} \times P_{i_{d-1}} \times \dotsb \times P_{i_1}.
\end{equation*}
The group $B^d$ acts to the right on $P_I$ by
\begin{equation}
(p_d,\ldots,p_1) \cdot (b_d,\ldots,b_1) =  (p_d b_d,b_d^{-1}p_{d-1} b_{d-1},\ldots,b_{2}^{-1}p_1 b_1).
\end{equation}
The Bott-Samelson manifold $BS^I$ is the quotient
\begin{equation}
BS^I = P_I/{B^d} = P_{i_d} \times^B P_{i_{d-1}} \times^B \dotsb \times^{B} P_{i_1} /B
\end{equation}
and is a complex projective variety of dimension $d$.  The orbit through a point $(p_d, \ldots, p_1) \in P_I$ is denoted by $[p_d, \ldots, p_1] \in BS^I$.

Let $I = [i_d, i_{d-1}, \ldots, i_1]$ and $I' = [i_{d-1}, \ldots , i_{11}]$. The map $P_I \to P_{i_d}$ given by the projection onto the first factor is equivariant with respect to the corresponding actions of $B^d$ and $B$ and therefore induces a map  $\pi \colon BS^I \to P_{i_d}/B \simeq \bbC P^1$
\begin{equation}\label{eq:oneproj}
\pi([p_{i_d}, p_{i_{d-1}}, \ldots, p_{i_d}]) = [p_{i_d}]
\end{equation}
with fiber $P_{I'}/B^{d-1} = BS^{I'}$; hence $BS^I$ is a $BS^{I'}$-bundle over $\bbC P^1$.

The group $P_{i_d}$ acts on $BS^I$ by left multiplication on the first factor; if $T$ is the maximal real torus included in $B$, then $T$ acts on the left on $Z_I$ by the action on the first factor:
\begin{equation}
t \cdot [p_d, \ldots, p_1] = [t p_d, \ldots, p_1].
\end{equation}
and the map \eqref{eq:oneproj} is $T$-equivariant. (It is, in fact, $P_{i_d}$-equivariant.)

There is a real/compact-group description of  $BS^I$, as follows: for every $j \in [n]$, we have $P_{j}/B \simeq K_j/T$ for a compact Lie group $K_j$ with $T \subset K_j \subset P_j$. The torus $T^d$ acts on 
\begin{equation*}
K_I = K_{i_d} \times \dotsb \times K_{i_1}
\end{equation*}
by
\begin{equation}
(k_d,\ldots,k_1) \cdot (t_d,\ldots,t_1) =  (k_d t_d,t_d^{-1}k_{d-1} t_{d-1},\ldots,t_{2}^{-1}k_1 t_1).
\end{equation}
Then 
\begin{equation*}
BS^I = K_I/{T^d} = K_{i_d} \times^T K_{i_{d-1}} \times^T \dotsb \times^{T} K_{i_1} /T. 
\end{equation*}
The orbit through  $(k_d, \ldots, k_1) \in K_I$ is denoted by $[k_d, \ldots, k_1] \in BS^I$.

We are interested in describing the assignment ring $\mathcal{A}_T(BS^I)$ for this action of $T$ on $BS^I$; more precisely, we want to construct generators for this $\bbS(\mf{t}^*)$-module and to understand the defect module $\mathcal{A}_T(BS^I)/\mathcal{H}_T(BS^I)$.

\subsection{Fixed-Point Data} \label{sec:fixpointdata}
For all indices $j \in [n]$, the fixed points of the $T$-action on $(P_j/B)^T \simeq (K_j/T)^T$ are given by
\begin{equation}
(K_j \cap N(T))/T \leqslant W,
\end{equation}
where $W$ is the Weyl group of $G$ and $N(T)$ is the normalizer of $T$. There are two such fixed points: one is $[1]$, the class of the identity, and the other is $[q_j]$, the class of the element $q_j \in K_j \cap N(T)$ whose co-adjoint action on $\ft^*$ is given by 
\begin{equation}
s_j :=s_{\alpha_j} \colon \ft^* \to \ft^*,
\end{equation}
the reflection determined by the simple root $\alpha_j$ (\cite{GHZ-2006}). Hence the fixed points for the $T$-action on $BS^I$ for $I= [j]$ are indexed by the set  $\{[0\cdot j], [1\cdot j] \} = \{[0], [j]\}$ of subwords of  $I$.
The weight of the (complex) isotropy representation of $T$ on the tangent space at $[0\cdot j]$ is $\alpha_j$, and the weight at $[1\cdot j]$ is $s_j(\alpha_j) = -\alpha_j$.

In general, the fixed-point set $(BS^I)^T$ is indexed by the subwords of $I$; these are expressions of the form 
\begin{equation}
J = J_{\epsilon} := [\epsilon_{1}i_1 , \ldots, \epsilon_{d} i_d]
\end{equation}
where $\epsilon \colon [d] \to \{0,1\}$ and $\epsilon_j := \epsilon(j)$. Hence, for general $I$, the $T$-action on $BS^I$ has $2^{|I|}$ fixed points. For a subword $J_{\epsilon}$ of $I$, the weights of the isotropy representation on the tangent space at the corresponding fixed point $p_\epsilon:= J_{\epsilon}$ are
\begin{equation}
s_{i_d}^{\epsilon_d}\alpha_{i_d}, \quad s_{i_{d}}^{\epsilon_d}s_{i_{d-1}}^{\epsilon_{d-1}}\alpha_{i_{d-1}}, \quad \ldots, \quad s_{i_d}^{\epsilon_d} \dotsb s_{i_1}^{\epsilon_1}\alpha_{i_1},
\end{equation}
where $s_j^1 = s_j$ and $s_j^0 = 1$, the identity transformation of $\mf{t}^*$.

\begin{rmk}
To save space we will sometimes remove the commas and brackets and denote, for example, $[0,1,2]$ by $012$.
\end{rmk}

\begin{exmp}\label{exm:bs21}
Let $I = [i_1, i_2] =  [2,1] =21$ be a word of length 2. The fixed point set $(BS^{21})^T$ is indexed by $\{00, 01, 20, 21\}$. The weights of the $T$-action are as follows:
\begin{align*}
\text{ at } 00:& \quad s_1^0\alpha_1 = \alpha_1 \quad \text{ and } \quad s_1^0s_2^0 \alpha_2 = \alpha_2\\
\text{ at } 01:& \quad s_1^1\alpha_1 = - \alpha_1 \quad \text{ and } \quad s_1^1s_2^0 \alpha_2 = s_1\alpha_2\\
\text{ at } 20:& \quad s_1^0\alpha_1 = \alpha_1 \quad \text{ and } \quad s_1^0s_2^1 \alpha_2 = -\alpha_2\\
\text{ at } 21:& \quad s_1^1\alpha_1 = -\alpha_1 \quad \text{ and } \quad s_1^1s_2^1 \alpha_2 = -s_1\alpha_2.
\end{align*}
\end{exmp}

\subsection{The Associated Graph}
The data in Example~\ref{exm:bs21} can be arranged in the following digraph (for each arrow, there is a second arrow in the reverse direction):
\begin{figure}[h]
\begin{tikzcd}[]
 01 \ar{r}[]{s_1\al_2} &
  21   \\
00   \ar{u}{\al_1} \ar{r}{\al_2} & 20   \ar{u}[swap]{\al_1}
\end{tikzcd}
$\Longrightarrow$
\begin{tikzcd}
1 \\
0 \ar{u}[swap]{\al_1}
\end{tikzcd}
\caption{The Graph of $BS^{[2,1]}$.}
\label{fig:1}
\end{figure}

The vertical arrows correspond to $P_1/B$; the horizontal arrows correspond to the fiber bundle $BS^{[2,1]} \to BS^{[1]}$: the bottom arrow corresponds to the fiber over $00$, which is a copy of $BS^{2}$, and the top arrow is the fiber over $01$, an $s_1$-twisted copy of $BS^{2}$. The weights of the opposite arrows are the opposite of the weights shown above: for example, the arrow $21 \to 20$ is labeled by $-\al_1$.

More general, to each Bott-Samelson space $BS^I$ we associate a decorated digraph $\Gamma_I$, with arrows labeled by roots,  constructed as follows:
\begin{itemize}
\item The vertices of $\Gamma_I$ are labeled by the $2^{|I|}$ subwords of $I$;
\item The edges of $\Gamma_I$ are of the form $J_{\epsilon}  J_{\epsilon'}$, for all subwords that differ  in exactly one position - for each edge there are two arrows, in both directions of the edge;
\item If $J_{\epsilon} = [\epsilon_{1}i_1 , \ldots, \epsilon_{d} i_d]$ and $J_{\epsilon'}$ differ exactly on position $j$, then the arrow $J_{\epsilon} \to J_{\epsilon'}$ is labeled by 
\begin{equation}
\alpha_{\epsilon, \epsilon'} = s_{i_d}^{\epsilon_d} \dotsb s_{i_j}^{\epsilon_j}\alpha_{i_j};
\end{equation}
note that $\alpha_{\epsilon', \epsilon} = -\alpha_{\epsilon, \epsilon'}$.
\end{itemize}

\begin{exmp} Let $I = [2, 1, 3]=213$; then $BS^{[2,1,3]} \to BS^{[3]}$ is a fiber bundle with fiber $BS^{[2,1]}$. The associated graph is:

\begin{figure}[H]
\begin{tikzcd}[back line/.style={densely dotted}, row sep=2em, column sep=2em]
& 
013  \ar{rr}{s_{3}s_1\al_2} 
  & & 
213   \\
003 \ar[crossing over]{rr}[near end]{s_{3}\al_2}  \ar{ur}{s_{3}\al_1}
  & & 
203 \ar{ur}[near start]{s_{3}\al_1}\\
& 
010 \ar[back line]{rr}[near start]{s_1\al_2} \ar[back line]{uu}[near start]{\al_3}
  & & 
210 \ar{uu}[near start]{\al_3}  \\
000 \ar[crossing over]{uu}{\al_3}    \ar[back line]{ur}{\al_1} \ar{rr}{\al_2} & & 
200 \ar[crossing over]{uu}[near start, swap]{\al_3}   \ar{ur}[swap]{\al_1}
\end{tikzcd}
\caption{The Graph of $BS^{[2,1,3]}$.}
\label{fig:figure121}
\end{figure}
Note that $BS^{[2,1,3]} \to BS^{[3]}$ is a bundle with fiber $BS^{[2,1]}$ (the bottom face, with vertices at points of the form $(*, *, 0)$, and $BS^{[2,1]} \to BS^{[1]}$ is a fiber bundle with fiber $BS^{[2]}$:
\begin{equation}
\xymatrix{
BS^{[2]} \ar[r] & BS^{[2,1]} \ar[d] \ar[r]& BS^{[2,1,3]} \ar[d]\\
& BS^{[1]}  &  BS^{[3]} } \,
\end{equation}
\end{exmp}

\section{Assignments for Bott-Samelson Spaces}

\subsection{Definitions}\label{sec:assigngraph}
The decorated digraph $\Gamma_I = (V_I, E_I, \alpha)$ encodes all the information needed to determine the polynomial 
assignments for the $T$-action on $BS^I$ (for details of the construction, see \cite{gsz2014}).
The $\bbS(\ft^*)$-algebra $\mathcal{A}_T(BS^I)$ is the $\bbS(\ft^*)$-subalgebra of $\text{Maps}(V_{I}, \bbS(\ft^*))$ consisting of those maps $f \colon V_{I} \to \bbS(\ft^*)$ with the property that
\begin{equation}
f(J_{\epsilon}) \equiv f(J_{\epsilon'}) \pmod{\alpha_{{\epsilon},{\epsilon'}}}
\end{equation}
for all edges $J_\epsilon J_{\epsilon'}$ of the graph $\Gamma_I$.

\begin{exmp}
If $I = [\, ]$ is the empty word, then $BS^{I} = \{ \text{pt} \}$ and $\mathcal{A}_T(BS^{[\,]}) = \bbS(\ft^*)$ is a free $\bbS(\ft^*)$-module, with a basis given by the assignment $A_{[\,]}$ that takes the value 1 at the fixed point $[\, ]$.
\end{exmp}

\begin{exmp}
If $I = [j]$ is a word of length 1, then $\mathcal{A}_T(BS^{[j]})$ consists of those maps $f \colon \{ [0], [j] \} \to \bbS(\ft^*)$ such that 
\begin{equation}
f([j]) \equiv f([0]) \pmod{\alpha_j}.
\end{equation} 
That is a free $\bbS(\ft^*)$-module, with a basis given by $\{ A_{[0]}, A_{[j]}\}$, where $A_{[0]}(J_{\epsilon}) = 1$ for all $J_{\epsilon}$ and 
\begin{equation}
A_{[j]} (J_{\epsilon})= 
\begin{cases}
0, & \text{ if } J_{\epsilon} = [0]\\
\alpha_j, & \text{ if }  J_{\epsilon} = [j]
\end{cases}.
\end{equation}
We record this information as a $2\times 2$ matrix
\begin{equation}
A_I = 
\begin{bmatrix}
1 & 0 \\ 1 & \alpha_j
\end{bmatrix}
\end{equation}
with columns corresponding to the assignments $A_{[0]}$ and $A_{[j]}$, and rows indexed by the fixed point set $(BS^{I})^T$, hence by the subwords of $I$.
\end{exmp}

\subsection{Cohomological Assignments}\label{sec:cohassign}
There is an injective morphism of $\bbS(\ft^*)$-algebras 
\begin{equation}
H_T(BS^I) \to \mathcal{A}_T(BS^I)
\end{equation}
from equivariant cohomology to assignments; therefore $H_T(BS^I)$ is identified with an $\bbS(\ft^*)$-subalgebra $\mathcal{H}_T(BS^I)$; assignments that are in the subalgebra $\mathcal{H}_T$ are called \emph{cohomological assignments}.

The subalgebra $\mathcal{H}_T(BS^I)$ is a \emph{free} $\bbS(\ft^*)$-module of rank $2^{|I|}$, with a basis given by cohomological assignments indexed by the fixed points, as follows: for $J = [\epsilon_1 i_1, \ldots, \epsilon_d i_d]$, let $H_J \colon (BS^I)^T \to \bbS(\ft^*)$ be defined by 
\begin{equation}
H_J(J') = 
\begin{cases}
(\alpha_{i_d}^{\epsilon_d})(s_{i_d}^{\epsilon_d'}\alpha_{i_{d-1}}^{\epsilon_{d-1}}) \dotsb (s_{i_d}^{\epsilon_d'}s_{i_{d-1}}^{\epsilon_{d-1}'}\dotsb s_{i_2}^{\epsilon'_2}\alpha_{i_{1}}^{\epsilon_{1}})
, & \text{ if } J \leqslant J' \\
0, & \text{ otherwise}
\end{cases}
\end{equation}
where $J \leqslant J' = [\epsilon_1' i_1, \ldots, \epsilon_d' i_d]$ if and only if $\epsilon_t \leqslant \epsilon_t'$ for all $t$. We record these classes in a matrix $H^I$ whose rows correspond to fixed points  and columns corresponding to classes, with  $H^I_{J',J} = H_J(J')$.

\begin{exmp} \label{exm:H21}
If $I = [2,1]$, then the matrix $H^{[2,1]}$ is given below:
\begin{equation*}
H^{[2,1]} = 
\kbordermatrix{
         &  H_{00} & H_{20} & H_{01} & H_{21}  \\
00  &   1 & 0 & 0 & 0 \\
20 &   1 & \al_2 & 0 & 0 \\
01 & 1 & 0 & \al_1 & 0 \\
21 & 1 & s_1\al_2 & \al_1 & \al_1s_1\al_2
} = \begin{bmatrix} 
H^{[2]} & 0 \\s_1H^{[2]} & \al_1 s_1H^{[2]}
\end{bmatrix}\; .
\end{equation*}
\end{exmp}

In general, if we list the fixed points in increasing right-to-left lexicographical order, the matrix $H^{[I,j]}$ can be computed from $H^{I}$ as follows:
\begin{equation}
H^{[I,j]} = \begin{bmatrix} 
H^{I} & 0 \\s_jH^{I} & \al_j s_jH^{I}
\end{bmatrix}\; .
\end{equation}

For $BS^{[2,1]}$ at each fixed point, the weights are pairwise non-collinear. Spaces that satisfy this condition are said to be of \emph{GKM type}.  

\begin{prop}
The $T$-space $BS^I$ is of GKM type if and only if the letters of $I$ are distinct.
\end{prop}

\begin{proof}
The weights at the fixed point $[0,0,\ldots, 0]$ are $\{\alpha_i\mid i \in I\}$, hence if $BS^I$ is of GKM type, then the letters of $I$ must be distinct.

For the converse, assume that $I = [i_1, \ldots, i_d]$ has distinct letters, but at some fixed point $J_{\epsilon} = [\epsilon_1 i_1, \ldots, \epsilon_d i_d]$, two of the weights are collinear. Then there exist $\ell<m$ in $[d]$ such that 
\begin{equation}
s_{i_d}^{\epsilon_d} \dotsb s_{i_m}^{\epsilon_m}\alpha_{i_m} = r s_{i_d}^{\epsilon_d} \dotsb s_{i_m}^{\epsilon_m} \dotsb s_{i_\ell}^{\epsilon_\ell}\alpha_{i_\ell},
\end{equation}
hence
\begin{equation}
\al_{i_m} =  r s_{i_{m-1}}^{\epsilon_{m-1}}\ldots s_{i_\ell}^{\epsilon_\ell} \al_{i_\ell} \in \text{span}\{\al_{i_\ell}, \al_{i_{\ell+1}}, \dotsb , \al_{i_{m-1}}\},
\end{equation}
which is impossible, since the simple roots are linearly independent. Hence if the weights at $[0, \ldots, 0]$ are pairwise non-collinear, so are the weights at any other fixed point.
\end{proof}

The only spaces $BS^I$ for which for which all assignments are cohomological are those of GKM type, hence those for which the letters of $I$ are distinct.

\subsection{Delta Classes}
What happens when $BS^I$ is not of GKM type? Then not all assignments are cohomological, and in this section we present examples of such classes, associated to points, edges, or faces of the graph $\Gamma_I$.

\begin{exmp} Let $I = [2, 1, 2]$; then $M = BS^{[2,1,2]} \to BS^{[2]}$ is a fiber bundle with fiber $BS^{[2,1]}$. The associated graph is:

\begin{figure}[H]
\begin{tikzcd}[back line/.style={densely dotted}, row sep=2em, column sep=2em]
& 
012  \ar{rr}{s_{2}s_1\al_2} 
  & & 
212   \\
002 \ar[crossing over]{rr}[near start]{s_{2}\al_2}  \ar{ur}{s_{2}\al_1}
  & & 
202 \ar{ur}[near start]{s_{2}\al_1}\\
& 
010 \ar[back line]{rr}[near start]{s_1\al_2} \ar[back line]{uu}[near start]{\al_2}
  & & 
210 \ar{uu}[near start]{\al_2}  \\
000 \ar[crossing over]{uu}{\al_2}    \ar[back line]{ur}{\al_1} \ar{rr}{\al_2} & & 
200 \ar[crossing over]{uu}[near start]{\al_2}   \ar{ur}[swap]{\al_1}
\end{tikzcd}
\caption{The Graph of $BS^{[2,1,2]}$.}
\label{fig:figure212}
\end{figure}
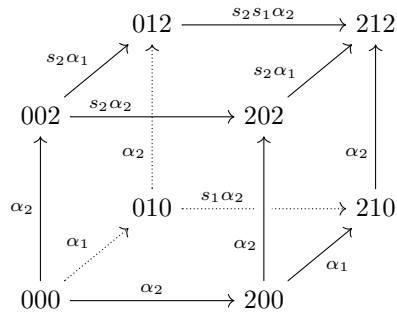

The weights at $000$ are $\al_1, \al_2, \al_2$ and are not pairwise non-collinear. Define $\delta = \delta_{000} \colon (BS^{[2,1,2]})^T \to \bbS(\ft^*)$,
\begin{equation}
\delta_{000} (v) = \begin{cases}  
\al_1 \al_2, & \text{ if } v = 000\\
0, & \text{ otherwise.}
\end{cases}
\end{equation}
Then $\delta$ is an assignment, but not a cohomological one: if it were, by the localization formula,
\begin{equation}
\int_M \delta = \frac{\al_1\al_2}{\al_1\al_2\al_2} = \frac{1}{\al_2},
\end{equation}
and this is not a polynomial; hence $\delta$ is a non-cohomological assignment.
\end{exmp}

More examples of delta classes, associated to edges and faces, are presented below. Some of them are cohomological, and some are not. For higher dimensional examples, the fact that the integral is a polynomial, while necessary, is not also sufficient to guarantee that a class is cohomological. (Necessary and sufficient conditions are given in \cite{p2014} and \cite{gsz2014}).

\begin{exmp}\label{exmp:1231} Let $M = BS^{[1,2,3,1]}$ and let $e$ be the 1-dimensional face (edge) with vertices at $1001$ and $1031$. A neighborhood of $e$ in the associated graph is shown below:

\begin{figure}[h]
\begin{tikzcd}[back line/.style={densely dotted}, row sep=1.5em, column sep=1.52em]
& & 0001 & & 
0031 & & \\
& &          & &          & & \\
1201 & & 
1001 \ar{uu}{\al_1} \ar{dd}[swap]{-\al_1} \ar{ll}[swap]{s_1\al_2} \ar{rr}{s_1\al_3}  & & 
1031\ar{rr}{s_1s_3\al_2} \ar{uu}[swap]{-s_1s_3\al_1} \ar{dd}{-\al_1} & & 1231 \\
& &          & &          & & \\
& & 
1000 & & 
10301 & &
\end{tikzcd}
\end{figure}
Then a class supported on the edge $e$ would assign polynomials of degree 3 to 1001 and 1031. However, if $\langle \alpha_3, \alpha_1\rangle = 0$ (for example, in type $A$),  then $s_1s_3\al_1 = -\al_1$ then a delta class $\delta_e$ given by
\begin{equation}
\delta_e(v) = 
\begin{cases}
\al_1 \, s_1\al_2, & \text{ if } v = 1001\\
\al_1 \, s_1s_3\al_2, & \text{ if } v = 1031\\
0, & \text{ otherwise}
\end{cases}
\end{equation}
assigns 0 or a polynomial of degree only two to each vertex.
\end{exmp}

\begin{exmp}\label{exmp:12312} Let $M = BS^{[1,2,3,1,2]}$ and let $f$ be the 2-dimensional face with vertices at $10010$, $10012$, $10312$, and $10310$. A neighborhood of $f$ in the associated graph is shown below:

\begin{figure}[h]
\begin{tikzcd}[back line/.style={densely dotted}, row sep=1.5em, column sep=1.5em]
10000 & & 
12010 & & 
12310 & & 
10300 \\
 & &  &   &  & & \\
00010 & & 
10010 \ar{uu}{s_1\al_2}  \ar{rr}{s_1\al_3} \ar{dd}{\al_2} \ar{ll}{\al_1} \ar{uull}{-\al_1} & & 
10310 \ar{uu}{s_1s_3\al_2} \ar{uurr}{-\al_1} \ar{rr}{-s_1s_3\al_1} \ar{dd}{\al_2} & & 
00310 \\
& &  & f &  & & \\
00012 & & 
10012 \ar{ll}{s_2\al_1} \ar{lldd}{-s_2\al_1} \ar{dd}{s_2s_1\al_2} \ar{rr}{s_2s_1\al_3}  & & 
10312 \ar{rr}{-s_2s_1s_3\al_1} \ar{dd}[swap]{s_2s_1s_3\al_2} \ar{ddrr}{-s_2\al_1} & & 
00312 \\
 & &  & &  & & \\
10002 & & 
12012 & & 
123120 & & 
10302 
\end{tikzcd}
\end{figure}
If $\langle \alpha_3, \alpha_1\rangle = 0$ (in type $A$, for example), then a delta class $\delta_f$ given by
\begin{equation}
\delta_f(v) = 
\begin{cases}
\al_1 \, s_1\al_2, & \text{ if } v = 10010\\
s_2\al_1 \, s_2s_1\al_2, & \text{ if } v = 10012\\
\al_1 \, s_1s_3\al_2, & \text{ if } v = 10310\\
s_2\al_1 \, s_2s_1s_3\al_2, & \text{ if } v=10312\\
0,  & \text{ otherwise}
\end{cases}
\end{equation}
is supported on the vertices of $f$ and assigns 0 or a polynomial of degree two to each vertex. Note that, although \emph{ a priori} a class supported on $f$ would be a class determined by three outside conditions (hence would assign a polynomial of degree three at each vertex), $\delta_f$ assigns a polynomial of degree two at each vertex: that is a consequence of the fact that at each vertex of $f$, the outside weights contain repeated directions.

\end{exmp}

\section{Inductive Construction of Assignments}

\subsection{Example} \label{sec:ex_inductive}
In this section we determine all the assignments on $M=BS^{[2,1,2]}$ and we show that $\mathcal{A}_T(M)$ is a free $\bbS(\ft^*)$-module. In subsequent sections, we expand the ideas of this construction to the general case.

The key remark is that $BS^{[2,1,2]} \to BS^{[2]}$ is a fiber bundle with fiber $BS^{[2,1]}$. On the associated graph, that is shown by the fact that the bottom face (the fiber over $[0]$) is a copy of $BS^{[2,1]}$, and the top face (the fiber over $[0]$) is an $s_2$-twisted copy of $BS^{[2,1]}$. Therefore, if $\eta \in \mathcal{A}_T(BS^{[2,1,2]})$, then
\begin{equation}
\eta|_{[.,0]} \in \mathcal{A}_T(BS^{[2,1]}) \quad \text{ and } \quad s_2\eta|_{[.,2]} \in \mathcal{A}_T(BS^{[2,1]}) .
\end{equation}

Let 
\begin{equation}
A = A^{[2,1]} = H^{[2,1]} = \begin{bmatrix}
1 & 0 & 0 & 0 \\
1 & \al_2 & 0 & 0 \\
1 & 0 & \al_1 & 0 \\
1 & s_1\al_2 & \al_1 & \al_1s_1\al_2
\end{bmatrix}
\end{equation}
be the matrix determined in Example~\ref{exm:H21}, with rows and columns indexed by $J \in \{[0,0], [2,0],[0,1],[2,1]\}$. Then
\begin{align*}
\eta([J,0]) & = a_{[0,0,0]} A_{[0,0]}(J) + a_{[2,0,0]} A_{[2,0]}(J) \\
& + a_{[0,1,0]} A_{[0,1]}(J) + a_{[2,1,0]} A_{[2,1]}(J)\\
\eta([J,2]) & = a_{[0,0,2]} s_2A_{[0,0]}(J) + a_{[2,0,2]} s_2A_{[2,0]}(J) \\
& + a_{[0,1,2]} s_2A_{[0,1]}(J) + a_{[2,1,2]} s_2A_{[2,1]}(J);
\end{align*}
with the action of $s_2$ on $\bbS(\ft^*)$ the extension of the action on $\ft^*$. Hence
\begin{equation}
\eta = 
\begin{bmatrix}
\eta([., 0]) \\ \eta([., 2])
\end{bmatrix}
= 
\begin{bmatrix}
A & 0 \\ 0 & s_2A
\end{bmatrix}
\begin{bmatrix}
a_{[., 0]} \\ a_{[., 2]}
\end{bmatrix}.
\end{equation}

Since $s_2 f \equiv f \pmod{\al_2}$ for all $f \in \bbS(\ft^*)$,  the remaining conditions for $\eta$ to be an assignment,
\begin{equation}
\eta([J,2]) \equiv \eta([J,0]) \pmod{\al_2},
\end{equation}
can be written as
\begin{equation}\label{eq:syst_al}
A 
\begin{bmatrix}
\Delta_{[0,0]} \\ \Delta_{[2,0]} \\ \Delta_{[0,1]} \\ \Delta_{[2,1]} 
\end{bmatrix} \equiv 0 \pmod{\al_2}, 
\end{equation}
where $\Delta_{J} = a_{[J,2]} - a_{[J,0]}$ for all $J \in \{[0,0], [2,0],[0,1],[2,1]\}$.

To determine all solutions of the system \eqref{eq:syst_al}, we first reduce the entries of $A$ modulo $\al_2$, then row-reduce the resulting matrix. In our case
\begin{equation}
A\mod{\al_2} = 
\begin{bmatrix}
1 & 0 & 0 & 0 \\
1 & 0 & 0 & 0 \\
1 & 0 & \al_1 & 0 \\
1 & -\langle \al_2, \al_1\rangle \al_1 & \al_1 & -\langle \al_2, \al_1\rangle \al_1^2
\end{bmatrix},
\end{equation}
where $\langle \al_2, \al_1\rangle$ is the corresponding Cartan integer. 

If $\langle \al_2, \al_1\rangle \neq 0$, then
\begin{equation}
\text{rref} (A\mod\al_2) = 
\begin{bmatrix}
1 & 0 & 0 & 0 \\
0 & 1 & 0 & \al_1 \\
0 & 0 & 1 & 0 \\
0 & 0 & 0 & 0 
\end{bmatrix}
\end{equation}

The general solution of the system \eqref{eq:syst_al} is then
\begin{equation}
\begin{bmatrix}
\Delta_{[0,0]} \\ \Delta_{[2,0]} \\ \Delta_{[0,1]} \\ \Delta_{[2,1]} 
\end{bmatrix} = 
t\begin{bmatrix}
0 \\ - \al_1 \\ 0 \\ 1
\end{bmatrix}
\end{equation}
with $t \in \bbS(\ft^*)$, arbitrary. Then $\eta \in \mathcal{A}_T(BS^{[2,1,2]})$ if and only if
\begin{align*}
a_{[0,0,0]} & = b_{[0,0,0]} &  a_{[0,0,2]} & = b_{[0,0,0]} + \al_2 b_{[0,0,2]} \\
a_{[0,1,0]} & = b_{[0,1,0]} &  a_{[0,1,2]} & = b_{[0,1,0]} + \al_2 b_{[0,1,2]} \\
a_{[2,1,0]} & = b_{[2,1,0]} &  a_{[2,1,2]} & = b_{[2,1,0]} + b_{[2,1,2]} \\
a_{[2,0,0]} & = b_{[2,0,0]} &  a_{[2,0,2]} & = b_{[2,0,0]} + \al_2 b_{[2,0,2]} - \al_1 b_{[2,1,2]}
\end{align*} 
with the $b$'s arbitrary polynomials in $\bbS(\ft^*)$. These relations can be written as
\begin{equation}
\begin{bmatrix}
a_{[., 0]} \\ a_{[., 2]} 
\end{bmatrix}
= 
\begin{bmatrix}
I_4 & 0 \\ I_4 & U^{[2,1,2]}
\end{bmatrix}
\begin{bmatrix}
b_{[., 0]} \\ b_{[., 2]} 
\end{bmatrix}
\end{equation}
with
\begin{equation}
U^{[2,1,2]} = \begin{bmatrix}
\al_2 & 0 & 0 & 0 \\
0 & \al_2 & 0 & -\al_1 \\
0 & 0 & \al_2 & 0 \\
0 & 0 & 0 & 1
\end{bmatrix}.
\end{equation}

Then 
\begin{equation}
\eta = 
\begin{bmatrix}
\eta([., 0]) \\ \eta([., 2])
\end{bmatrix}
= 
\begin{bmatrix}
A & 0 \\ 0 & s_2A
\end{bmatrix}
\begin{bmatrix}
a_{[., 0]} \\ a_{[., 2]}
\end{bmatrix} = \begin{bmatrix}
A &  0 \\ s_2A & s_2A\cdot U^{[2,1,2]}
\end{bmatrix}
\begin{bmatrix}
b_{[., 0]} \\ b_{[., 2]} 
\end{bmatrix}
\end{equation}
hence every assignment on $BS^{[2,1,2]}$ can be written as a linear combination of the columns of the matrix
\begin{equation*}
\begin{bmatrix}
1 & 0 & 0 & 0 & 0 & 0 & 0 & 0 \\
1 & \al_2 & 0 & 0 & 0 & 0 & 0 & 0 \\
1 & 0 & \al_1 & 0 & 0 & 0 & 0 & 0 \\
1 & s_1\al_2 & \al_1 & \al_1s_1\al_2& 0 & 0 & 0 & 0 \\
1 & 0 & 0 & 0 & \al_2 & 0 & 0 & 0 \\
1 & -\al_2 & 0 & 0 & \al_2&  -\al_2^2& 0 & \al_1\al_2 \\
1 & 0 & s_2\al_1 & 0 & \al_2 & 0 & \al_2s_2\al_1 & 0 \\
1 & s_2s_1\al_2 & s_2\al_1 & s_2\al_1 s_2s_1\al_2& \al_2 & \al_2s_2s_1\al_2 & \al_2s_2\al_1 &  - \langle \al_1, \al_2 \rangle \al_2s_2s_1\al_2
\end{bmatrix} 
\end{equation*}

These columns are linearly independent, hence they form a basis of the $\bbS(\ft^*)$-module $\mathcal{A}_T(BS^{[2,1,2]})$; therefore one can use 
\begin{equation}
A^{[2,1,2]} = 
\begin{bmatrix}
A^{[2,1]} & 0 \\ s_2A^{[2,1]} & s_2A^{[2,1]}\cdot U^{[2,1,2]}
\end{bmatrix}
\end{equation}
as a matrix recording a basis for $\mathcal{A}_T(BS^{[2,1,2]})$. Note that
\begin{equation}
H^{[2,1,2]} = 
\begin{bmatrix}
H^{[2,1]} & 0 \\ s_2H^{[2,1]} & \al_2 s_2H^{[2,1]}
\end{bmatrix}
\end{equation}
would correspond to the situation $U = \al_2 I$.

The first seven columns of $A^{[2,1,2]}$ are the same as the first seven columns of $H^{[2,1,2]}$; in particular, those seven assignments are cohomological. The eight column, $(A^{[2,1,2]})_{[2,1,2]}$, is different, because $U \neq \al_2 I_4$. One can check that
\begin{equation}
(A^{[2,1,2]})_{[2,1,2]} =  - \langle \al_1, \al_2 \rangle (A^{[2,1,2]})_{[2,0,2]}  + \delta_{[2,0,2]}.
\end{equation}
In particular, one can replace the eight assignment class by $\delta_{[2,0,2]}$. This proves that, for $M = BS^{[2,1,2]}$, the defect module $\mathcal{A}_T(M)/\mathcal{H}_T(M)$ is generated by a delta assignment,  $\delta_{[2,0,2]}$: every assignment on $M$ can be written as a combination of a cohomological assignment and $\delta_{[2,0,2]}$.

\subsection{General Construction} \label{sec:gen_construction}
The construction in the general case follows the same procedure.

Let $I$ be a word of length $|I| = d$, $I' = [I, j]$, and $\alpha = \alpha_j$. Then the projection on the last letter, 
\begin{equation}
\pi_j \colon \Gamma_{]I,j]} \to \Gamma_{[j]}
\end{equation}
is the combinatorial description of the fiber bundle $BS^{[I,j]} \to BS^{[j]}$; the fibers are
\begin{align*}
\Gamma_{[I,j]}^{[0]} & := \pi_j^{-1}([0]) \simeq \Gamma_I \\
\Gamma_{[I,j]}^{[j]} & := \pi_j^{-1}([j]) \simeq s_j\Gamma_I .
\end{align*}

Suppose the assignment module $\mathcal{A}_T(BS^I)$ is a free $\bbS(\ft^*)$-module of rank $2^{d}$, with a basis indexed by the subwords of $I$. Let $A = A^{I}$ be the matrix whose columns record the vectors of this basis.

If $\eta \in \mathcal{A}_T(BS^{I'})$, then the restrictions to the two fibers are assignments on the fibers, hence
\begin{align*}
\eta([., 0]) \in \mathcal{A}_T(BS^I) & \Longrightarrow \eta([., 0]) = \sum_J a_{[J,0]} A_J(.) \\
s_j\eta([., j]) \in \mathcal{A}_T(BS^I) & \Longrightarrow \eta([., j]) = \sum_J a_{[J,j]} s_jA_J(.) ,
\end{align*}
with the summations over the subwords $J$ of $I$. Hence 
\begin{equation}
\eta = 
\begin{bmatrix}
\eta([., 0]) \\ \eta([., j])
\end{bmatrix}
 = 
 \begin{bmatrix}
 A & 0 \\ 0 & s_jA
 \end{bmatrix}
 \begin{bmatrix}
 a_[., 0] \\ a_[., j]
 \end{bmatrix}.
\end{equation}

The compatibility conditions along the edges $[J, 0] - [J, j]$ are satisfied if and only if
\begin{equation}\label{eq:syst_gen}
A [ \Delta_{[.]} ] \equiv 0 \pmod{\al},
\end{equation}
where $\Delta_{[.]} = a_{[., j]} - a_{[., 0]}$.

A second, essential, assumption is that the matrix $A\mod{\alpha}$ can be transformed into a reduced row echelon form through row operations and column swapping - this is not a trivial assumption, since we are operating with matrices of polynomials, hence over the \emph{ring} $\bbS(\ft^*)$. If $J'$ are the indices of the basic variables and $J''$ the indices of the free variables, we assume that we have written the columns of $A$ in that order;  then $A$ can be transformed into 
\begin{equation}
\Phi_{\alpha}(A) = 
\begin{bmatrix}
I & C \\ 0 & 0 
\end{bmatrix},
\end{equation}
Then the general solution of the system of congruences \eqref{eq:syst_gen} is 
\begin{equation}
\begin{bmatrix}
a_{[.,0]} \\ a_{[J', j]} \\ a_{[J'', j]}
\end{bmatrix}
=
\begin{bmatrix}
I & 0 \\
I & 
\begin{bmatrix}
\alpha I & -C \\ 0 & I
\end{bmatrix}
\end{bmatrix}
\begin{bmatrix}
b_{[.,0]} \\ b_{[J', j]} \\ b_{[J'', j]}
\end{bmatrix}.
\end{equation}

Let 
\begin{equation}
U^{[I,j]} = \begin{bmatrix}
\alpha I & -C \\ 0 & I
\end{bmatrix};
\end{equation}
then
\begin{equation}
\eta = 
\begin{bmatrix}
\eta([., 0]) \\ \eta([., j])
\end{bmatrix}
 = 
 \begin{bmatrix}
 A & 0 \\ s_{\alpha}A  & s_\alpha A\cdot U^{[I,j]}
 \end{bmatrix}
 \begin{bmatrix}
 b_[., 0] \\ b_[., j]
 \end{bmatrix},
\end{equation}
hence every assignment for $BS^{[I,j]}$ can be written as a linear combination of the columns of the matrix
\begin{equation}
A^{[I,j]} = 
\begin{bmatrix}
A^I & 0 \\ s_j A^I & s_j A^I \cdot U^{[I,j]}
\end{bmatrix}.
\end{equation}

Since the columns of $A^I$ form a basis for $\mathcal{A}_T(BS^I)$, the matrix $A^I$ is invertible over the field of fractions of $\bbS(\ft^*)$. Then $s_j A^I$ is invertible, and since $U$ is also invertible, it follows that $A^{[I,j]}$ is invertible over the field of fractions of $\bbS(\ft^*)$. Therefore the columns of $A^{[I,j]}$ are independent over $\bbS(\ft^*)$, hence the columns of $A^{[I,j]}$ form a basis of $\mathcal{A}_T(BS^{[I,j]})$.

\subsection{Transition Formula}
We have seen that if $I$ has distinct letters, then $BS^I$ is of GKM-type and all assignments are cohomological; hence we can take $A^I = H^I$. In, general, suppose 
\begin{equation}
A^I = H^I V^I
\end{equation}
with $V^I$ a matrix with entries in $Q(\ft^*)$, the field of fractions of $\bbS(\ft^*)$. Then, under the assumptions of the previous section,
\begin{align*}
A^{[I,j]} = 
\begin{bmatrix}
A^I & 0 \\ s_j A^I & s_j A^I \cdot U^{[I,j]}
\end{bmatrix}
= 
\begin{bmatrix}
H^I & 0 \\ s_j H^I & \al_j s_j H^I 
\end{bmatrix}
\begin{bmatrix}
V^I & 0 \\ \partial_j V^I & \frac{1}{\al_j} s_jV^I \cdot U^{[I,j]}
\end{bmatrix},
\end{align*}
where 
\begin{equation}
\partial_j (f) = \frac{1}{\al_j} (s_j f - f)
\end{equation}
is the divided difference operator, extended to matrices with entries in $Q(\ft^*)$. Therefore
\begin{equation}
A^{[I,j]} = H^{[I,j]} V^{[I,j]},
\end{equation}
where
\begin{equation}
V^{[I,j]} = %
\begin{bmatrix}
V^I & 0 \\ \partial_j V^I & \frac{1}{\al_j} s_jV^I \cdot U^{[I,j]}
\end{bmatrix}.
\end{equation}
In particular, under our assumptions, each assignment can be written in a unique way as a linear combination \emph{over the field of fractions} $Q(\ft^*)$  of the cohomological assignments that occur as columns of $H$; the cohomological assignments are those for which the coefficients are polynomials. 

For example, the delta assignment $\delta_e$ on $BS^{[1,2,3,1]}$ in Example~\ref{exmp:1231} can be written as
\begin{equation}
\delta_e = -\frac{\al_1+\al_2}{\al_1} H_{1001} + \frac{1}{\al_1} H_{1201} - \frac{1}{\al_1} H_{1031};
\end{equation}
hence $\delta_e$ is not cohomological. Similarly, the delta assignment $\delta_f$ for $BS^{[1,2,3,1,2]}$ in Example~\ref{exmp:12312} can be written as a linear combination of $H_{10010}$, $H_{10310}$, $H_{10012}$, $H_{12012}$, and $H_{10312}$, but not all coefficients are polynomials (in fact, none is), so $\delta_f$ is non cohomological.

\subsection{Software}
We have extensively used the Coxeter package written by John Stembridge and MAPLE $15$ for our calculations. Using notations from the last section our input is the Lie-type of the Lie group and the word $I^{\prime} =[I,\alpha]$ of length $d+1$ that defines the Bott-Samelson space  $BS^{I^{\prime}}$. Using the Cartan matrix for the given Lie-type and simple reflections from Coxeter we next generate the one-skeleton of $BS^{I^{\prime}}$. We store it as a $2^{d+1} \times (d+1)$ matrix $M$ whose $(i,j)$ entry is a list of the form $[v_i,v_j,\alpha_{i,j}]$. Here $v_i$ is the $i^{th}$ fixed point, $v_j$ is the $j^{th}$ neighboring fixed point of $v_i$ and $\alpha_{i,j}$ is the weight vector pointing from $v_i$ towards $v_j$.  We initialize the iteration by starting from $[1]$ as a basis of assignment cohomology of a point as a module over $S(\mathfrak{t}^*)$. Then, in each iteration step we row reduce the basis matrix from the previous step such that while performing the reduction step we stay inside the polynomial ring $S(\mathfrak{t}^*)$. Therefore, using the general construction outlined in the last section we produce the two matrices 
\begin{eqnarray*}
\Phi_{\alpha}(A) = \begin{bmatrix} I & C \\ 0 & 0   \end{bmatrix}, U = \begin{bmatrix}  \alpha I & -C \\ 0 & I                 \end{bmatrix}.   
\end{eqnarray*}
The basis matrix $A^{I^{\prime}}$ is then given by
\begin{eqnarray*}
A^{I^{\prime}} = \begin{bmatrix}   A^I & 0 \\ s_{\alpha}A^I & s_{\alpha}A^I.U   \end{bmatrix}.
\end{eqnarray*}

\section{Morse-Type Generators}

\subsection{Generating Ideals}
Let $BS^I$ be a Bott-Samelson space, associated to a word $I$. A vector $\xi \in \ft$ is called \emph{polarizing} if $\alpha(\xi) \neq 0$ for all roots; let $\mathcal{P} \subset \ft$ be the set of all polarizing vectors.

For a polarizing vector $\xi$, we say that an edge $(p,q)$ of the graph $\Gamma_I$ is \emph{ascending} if $\alpha_{p,q}(\xi) > 0$, and is \emph{descending} if $\alpha_{p,q}(\xi) < 0$.

The Cartan-Killing form  on $\ft$ allows us to canonically identify $\ft$ and $\ft^*$: for $\beta \in \ft^*$, there is a unique $X_{\beta} \in \ft$ such that $\beta(\xi) = X_{\beta} \cdot \xi$. Moreover, $\alpha \cdot \beta = X_{\alpha}\cdot  X_{\beta}$ for all $\alpha, \beta \in \ft^*$.

If $\alpha, \beta$ are roots, then
\begin{align*}
s_{\alpha}\beta(\xi) & = \beta(\xi) - 2\frac{\beta \cdot \alpha}{\alpha \cdot  \alpha} \alpha(\xi) = X_{\beta}\cdot \xi -  2\frac{X_\beta \cdot X_\alpha}{X_\alpha \cdot X_\alpha}X_{\alpha}\cdot \xi \\
& = X_\beta \cdot \Bigl(  \xi - 2\frac{X_\alpha \cdot \xi}{X_\alpha \cdot X_\alpha}X_{\alpha}\Bigr) = \beta(s_{X_\alpha}\xi);
\end{align*}
in particular, if $\xi$ is polarizing, then $s_{X_\alpha}\xi$ is also polarizing.

\begin{prop}
For every choice of a polarizing vector $\xi$, if all the edges of $\Gamma_I$ are oriented as ascending edges, then the resulting oriented graph is acyclic.
\end{prop}

\begin{proof}
We prove the statement by induction on $|I|$. 

The base case $I = [i]$ is clear. Suppose the statement is true for $I$ and let $I' = [I, j]$ be an extension of $I$. Let $BS^{[I,j]} \to BS^{[j]}$ be the fiber bundle over $\bbC P^1$ with fiber $BS^I$. Any cycle of $\Gamma_{[I,j]}$ lies either in the fiber over $[0]$ or in the fiber over $[j]$, since all the arrows from one fiber to the other are oriented the same way. The fiber over $[0]$ is a copy of $BS^{I}$ and, by induction, the corresponding graph is acyclic. The fiber over $[j]$ is an $s_{j}$-twisted copy; the reflection $s_j$ induces an isomorphism of \emph{oriented} graphs between this copy and the original $BS^I$, the former oriented by $\xi$ and the latter 
by $s_{X_\alpha}\xi$. Hence the fiber over $[j]$ is also acyclic.
\end{proof}

Let $BS^{I}$ be a Bott-Samelson space and $\xi \in \mathcal{P}$ a polarizing vector.

We define a partial order on the vertices of $\Gamma_I$ (fixed points of the $T$-action on $BS^I$) by $p \preccurlyeq q$ if there exists a chain of ascending edges from $p$ to $q$ in the graph $\Gamma_I$, oriented by the polarizing vector $\xi$.

\begin{exmp}
For example, for $BS^{[2,1]}$, a polarizing $\xi$ in the positive Weyl chamber orients the graph $\Gamma_{[2,1]}$ as follows (the arrows indicate the ascending directions): 
\begin{figure}[h]
\begin{tikzcd}[]
 01 \ar{r}[]{s_1\al_2} &
21   \\
00   \ar{u}{\al_1} \ar{r}{\al_2} & 
20   \ar{u}[swap]{\al_1}
\end{tikzcd}
\end{figure}

The partial order has two maximal chains, $(0,0) \prec (0,1) \prec (2,1)$ and $(0,0) \prec (2,0) \prec (2,1)$; the vertices $(0,1)$ and $(2,0)$ are incomparable.
\end{exmp}

For a vertex $p$ of $\Gamma_I$, define the \emph{flow-up}
\begin{equation}
\mathcal{F}_p = \{ q \mid p \preccurlyeq q\}
\end{equation}
and let 
\begin{equation}
\mathcal{A}_T(BS_{\succcurlyeq p}^I)) = \{ \eta \in \mathcal{A}_T(BS^I)\mid \eta(q) = 0 \text{ if }  q \not\in \mathcal{F}_p\}
\end{equation}
be the subalgebra of assignments supported on $\mathcal{F}_p$. Let
\begin{equation}
\mathcal{I}_{(p)}= \{ \eta(p) \mid \eta \in \mathcal{A}_T(BS_{\succcurlyeq p}^I)\} \subset \bbS(\ft^*)
\end{equation}
be the ideal of $\bbS(\ft^*)$ defined by the values at $p$ of assignments supported on the flow-up from $p$. Then $\mathcal{I}_{(p)}$ is finitely generated, because $\bbS(\ft^*)$ is Noetherian.

For a vertex $p$, let $\mathcal{G}_p$ be a finite set of assignments in $\mathcal{A}_T(BS_{\succcurlyeq p}^I)$ such that the values at $p$ of the assignments in $\mathcal{G}_p$ generate the ideal $\mathcal{I}_{(p)}$ over $\bbS(\ft^*)$:
\begin{equation}
\mathcal{I}_{(p)}  = \langle \eta(p) \mid \eta \in \mathcal{G}_p \rangle_{\bbS(\ft^*)}.
\end{equation}
Then
\begin{equation}
\mathcal{G} = \bigcup_p \mathcal{G}_p 
\end{equation}
is a system of generators for $\mathcal{A}_T(BS^I)$; we will refer to these assignments as \emph{Morse generators}.

\subsection{Construction of Morse Generators}
Let $p$ be a fixed point for the $T$-action on $BS^I$, and let $\xi$ be a polarizing vector. Suppose $\mathcal{A}_T(BS^I)$ is a free $\bbS(\ft^*)$-module, with a basis recorded as columns of a matrix $A^I$; each column $A_J^{I}$ corresponds to a subword $J \subseteq I$; hence there are $n = 2^{|I|}$ columns.
Let $\{ v_1, v_2, \ldots, v_k\}$ be the complement of the flow-up $\mathcal{F}_p$.

Let $\eta \in \mathcal{A}_T(BS^I_{\succcurlyeq p})$ be an assignment supported on the flow-up $\mathcal{F}_p$. Then
\begin{equation}
\eta = \sum_J c^J A_J^I
\end{equation}
with $[c_J] \in \bbS(\ft^*)^{n}$. The condition $\eta \in \mathcal{A}_T(BS^I_{\succcurlyeq p})$ is equivalent to 
\begin{equation}
\sum_J c^J A_{J}^I(v_j) = 0
\end{equation}
for all $1 \leqslant j \leqslant k$. Let $F_J = [A_J^I(v_j)]_j \in \bbS(\ft^*)^k$ be the subcolumn of $A_J^I$ corresponding to rows in $\{v_1, \ldots, v_k\}$ and let 
\begin{equation}
Y = \langle F_J \mid J\subseteq I \rangle_{\bbS(\ft^*)} \subseteq \bbS(\ft^*)^k
\end{equation}
be the submodule of $\bbS(\ft^*)^k$ generated by the $F_J$'s. Then $\eta \in \mathcal{A}_T(BS^I_{\succcurlyeq p})$ if and only if 
\begin{equation}
[c^J]_J \in Syz(\langle F_J \mid J\subseteq I \rangle) \subseteq \bbS(\ft^*)^n, 
\end{equation}
the syzygy module of $Y$. Suppose that 
\begin{equation}
Syz(\langle F_J \mid J\subseteq I \rangle) = \langle g_1, g_2, \ldots, g_t \rangle_{\bbS(\ft^*)}
\end{equation}
is generated by $g_j \in \bbS(\ft^*)^n$, $1\leqslant j \leqslant t$. Define 
\begin{equation}
\eta_j^{(p)} = \sum_J (g_j)^J A_J^{I}.
\end{equation}
Then 
\begin{equation}
\mathcal{A}_T(BS_{\succcurlyeq p}^I)) = \langle \eta_j^{(p)} \mid j = 1, \ldots, t \rangle_{\bbS(\ft^*)}.
\end{equation}
Let 
\begin{equation}
r_j = \eta_j^{(p)}(p)
\end{equation}
for all $1\leqslant j \leqslant t$ and suppose the nonzero values are $r_1, \ldots, r_s$.

\begin{prop}
The ideal $\mathcal{I}_{(p)}$ is generated by $r_1, \ldots, r_s$:
\begin{equation}
\mathcal{I}_{(p)} = \langle r_1, r_2, \ldots, r_s \rangle_{\bbS(\ft^*)}.
\end{equation} 
\end{prop}

\begin{proof}
Let $f \in \mathcal{I}_{(p)} $ and $\eta \in \mathcal{A}_T(BS^I_{\succcurlyeq p})$ such that $f = \eta(p)$. Then
\begin{equation}
\eta = \sum_J c^J A_J^I
\end{equation}
and 
\begin{equation}
[c^J]_J \in Syz(\langle F_J \mid J\subseteq I \rangle)  =  \langle g_1, g_2, \ldots, g_t \rangle_{\bbS(\ft^*)}.
\end{equation}
Then
\begin{equation}
[c^J]_J = \sum_{j=1}^t a_j g^j
\end{equation}
for polynomials $a_j \in \bbS(\ft^*)$, and therefore 
\begin{equation}
\eta = \sum_J \sum_{j=1}^t a_j (g^j)^J A_J^I = \sum_{j=1}^t a_j \eta_j^{(p)},
\end{equation}
hence 
\begin{equation}
f = \eta(p) = \sum_{j=1}^t a_j \eta_j^{(p)} (p) = \sum_{j=1}^t a_j r_j = \sum_{j=1}^s a_j r_j,
\end{equation}
which proves that $r_1, \ldots, r_s$ generate $\mathcal{I}_{(p)}$.
\end{proof}

\subsection{Example} 

In this section we show how the general construction described in the previous section works in a particular example.

\begin{exmp}
Consider the Bott-Samelson manifold $BS^{[2,1,2]}$ in Lie-type $A_2$ and the generic polarization given by $\xi$ in the positive Weyl chamber. The polarized graph $\Gamma$ is

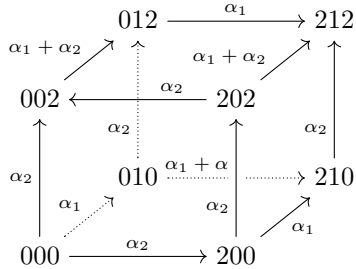
\begin{figure}[h]
\begin{tikzcd}[back line/.style={densely dotted}, row sep=1.5em, column sep=1.5em]
& 
012  \ar{rr}{\al_1}
  & & 
212   \\
002  \ar{ur}{\al_1+\al_2}
  & & 
202  \ar{ll}[near start, swap]{\al_2}   \ar{ur}[near start]{\al_1+\al_2}\\
& 
010 \ar[back line]{rr}[near start]{\al_1+\al_2} \ar[back line]{uu}[near start]{\al_2}
  & & 
210 \ar{uu}[near start]{\al_2}  \\
000 \ar[crossing over]{uu}{\al_2}    \ar[back line]{ur}{\al_1} \ar{rr}{\al_2} & & 
200 \ar[crossing over]{uu}[near start]{\al_2}   \ar{ur}[swap]{\al_1}
\end{tikzcd}
\caption{The Graph of $BS^{[2,1,2]}$.}
\label{fig:figure212-2}
\end{figure}

A basis for the assignment module is given by the columns of the matrix
\begin{equation*}
\begin{bmatrix}
1 & 0 & 0 & 0 & 0 & 0 & 0 & 0 \\
1 & \al_2 & 0 & 0 & 0 & 0 & 0 & 0 \\
1 & 0 & \al_1 & 0 & 0 & 0 & 0 & 0 \\
1 & \al_1+\al_2 & \al_1 & \al_1(\al_1+\al_2)& 0 & 0 & 0 & 0 \\
1 & 0 & 0 & 0 & \al_2 & 0 & 0 & 0 \\
1 & -\al_2 & 0 & 0 & \al_2&  -\al_2^2& 0 & \al_1\al_2 \\
1 & 0 & s_2\al_1 & 0 & \al_2 & 0 & \al_2(\al_1+\al_2) & 0 \\
1 & \al_1 & \al_1+\al_2 & \al_1(\al_1+\al_2) & \al_2 & \al_1\al_2 & \al_2(\al_1+\al_2) &   \al_1\al_2
\end{bmatrix} 
\end{equation*}
the rows and columns are indexed by subwords of $[2,1,2]$, in right-to-left lexicographic order
\begin{equation*}
[0,0,0], [2,0,0], [0,1,0], [2,1,0], [0,0,2], [2,0,2], [0,1,2], [2,1,2]
\end{equation*}
We also write the generators as rows, with columns corresponding to the values at the fixed points:

\begin{equation*}
\kbordermatrix{
& 000 & 200 & 010 & 210 & 002 & 202 & 012 & 212\\
\eta_1^{000} & 1 & 1 & 1 & 1 & 1 & 1 & 1 & 1 \\
\eta_1^{200} & 0 & \al_2 & 0  & \al_1 + \al_2 & 0  & -\al_2 & 0  & \al_1 \\
\eta_1^{010} & 0 & 0  & \al_1 & \al_1 & 0 & 0 & \al_1+\al_2 & \al_1 + \al_2\\
\eta_1^{210} & 0 & 0 & 0 & \al_1(\al_1+\al_2) & 0 & 0 & 0 & \al_1(\al_1+\al_2)\\
\eta_1^{002} & 0 & 0 & 0 & 0 & \al_2^2 & 0 & \al_2^2 & \al_2(\al_1+\al_2)\\
\eta_2^{002} & 0 & 0 & 0 & 0 & -\al_1\al_2 & 0 & -\al_1\al_2 & 0 \\
\eta_1^{202} & 0 & 0 & 0 & 0 & \al_2 & \al_2 & \al_2 & \al_2 \\
\eta_2^{202} & 0 & 0 & 0 & 0 & 0 & -\al_2^2 & 0 & \al_1\al_2\\
\eta_3^{202} & 0 & 0 & 0 & 0 & 0 & \al_1\al_2 & 0 & \al_1\al_2\\
\eta_1^{012} & 0 & 0 & 0 & 0 & 0 & 0 & \al_2(\al_1+\al_2) & \al_2(\al_1+\al_2) \\
\eta_1^{212} & 0 & 0 & 0 & 0 & 0 & 0 & 0 & \al_1\al_2(\al_1+\al_2)
}
\end{equation*}
These generators are not independent over $\bbS(\ft^*)$; the dependencies are:
\begin{align*}
\eta_2^{202} & =\eta_1^{002} -\al_2 \eta_1^{202}  \\
\eta_3^{202} & = \eta_2^{002} +\al_1 \eta_1^{202}  \\
\eta_1^{212} & = \al_1 \eta_1^{002} + \al_2 \eta_2^{002}.
\end{align*}
We can then eliminate $\eta_2^{202}$, $\eta_3^{202}$, and $\eta_1^{212}$ and obtain a basis.  The generating ideals $\mathcal{I}_J$ are as follows:
\begin{align*}
\mathcal{I}_{(000)} & = \langle 1 \rangle & \mathcal{I}_{(002)}   & =   \langle \al_1\al_2, \al_2^2 \rangle\\
\mathcal{I}_{(200)} & = \langle \al_2\rangle &\mathcal{I}_{(202)} & = \langle \al_2 \rangle\\
\mathcal{I}_{(010)} & = \langle \al_1\rangle &\mathcal{I}_{(012)} & = \langle \al_2(\al_1+\al_2) \rangle\\
\mathcal{I}_{(210)} & = \langle \al_1(\al_1+\al_2)\rangle &\mathcal{I}_{(212)}& = \langle \al_1\al_2(\al_1+\al_2) \rangle
\end{align*} 
Except for $\mathcal{I}_{(002)}$, all other ideals are principal; notice that $002$ is the highest vertex with repeated directions of the weights. That is not a coincidence, as we will show in a separate project \cite{gtz201x}.
\end{exmp}

\bibliographystyle{alpha}

\end{document}